\documentclass[conference]{IEEEtran}
\IEEEoverridecommandlockouts
\usepackage{cite}
\usepackage{amsmath, amssymb, amsfonts}
\usepackage{algorithmic}
\usepackage{graphicx}
\usepackage{textcomp}
\usepackage{xcolor}
\def\BibTeX{{\rm B\kern-.05em{\sc i\kern-.025em b}\kern-.08em
    T\kern-.1667em\lower.7ex\hbox{E}\kern-.125emX}}

\usepackage{tikz}
\usepackage{algorithm}
\usepackage{amsthm}
\newtheorem*{prop}{Proposition}
\newtheorem*{cor}{Corollary}
\newtheorem{rem}{Remark}

\begin{document}

\title{Neural Multigrid Architectures}

\author{\IEEEauthorblockN{Vladimir Fanaskov}
\IEEEauthorblockA{\textit{Center for Design, Manufacturing, and Materials } \\
\textit{Skolkovo Institute of Science and Technology}\\
Moscow, Russia \\
vladimir.fanaskov@skoltech.ru}}

\maketitle

\begin{abstract}
We propose a convenient matrix-free neural architecture for the multigrid method. The architecture is simple enough to be implemented in less than fifty lines of code, yet it encompasses a large number of distinct multigrid solvers. We argue that a fixed neural network without dense layers can not realize an efficient iterative method. Because of that, standard training protocols do not lead to competitive solvers. To overcome this difficulty, we use parameter sharing and serialization of layers. The resulting network can be trained on linear problems with thousands of unknowns and retains its efficiency on problems with millions of unknowns. From the point of view of numerical linear algebra network's training corresponds to finding optimal smoothers for the geometric multigrid method. We demonstrate our approach on a few second-order elliptic equations. For tested linear systems, we obtain from two to five times smaller spectral radius of the error propagation matrix compare to a basic linear multigrid with Jacobi smoother.
\end{abstract}

\section{Introduction}

In this article, we describe how neural networks can be used to solve a system of linear equations
\begin{equation}
    \label{nmg:eq:linear_equation}
    A x = b,~x\in\mathbb{R}^{n},~b\in\mathbb{R}^{n},~A\in\mathbb{R}^{n\times n},
\end{equation}
in a particular case when $A$ results from the discretization of PDE. Since $A$ is typically large and sparse, iterative methods are preferable to direct ones \cite[Section 8]{Saad_hist}. It is known that arbitrary linear iterative method has a form
\begin{equation}
    \label{nmg:eq:linear_iteration}
    x^{(m+1)} = x^{(m)} + N\left(b - Ax^{(m)}\right),~\det(N)\neq0,
\end{equation}
where $N$ is an approximate inverse of $A$, and $m$ is iteration number \cite[Section 2.2.2]{Hackbusch_iter}. To solve equation~\eqref{nmg:eq:linear_equation}, we represent $N$ as linear neural network $\mathcal{N}(\omega)$ and tune parameters $\omega$ to improve convergence speed, i.e., to obtain as small spectral radius of error propagation matrix $I - NA$ as possible.

The architecture of $\mathcal{N}(\omega)$ is chosen so that it corresponds to a particular geometric multigrid solver. For the article to be self-consistent, we provide a brief review of multigrid techniques in Section~\ref{nmg:section:multigrid_method}. The resulting network $\mathcal{N}(\omega)$ consists of convolutional layers and does not depend on a matrix of a linear operator so that it can be conveniently implemented and applied to a variety of linear problems. The architecture can be found in Section~\ref{nmg:section:multigrid_neural}.

The loss function that we use for unsupervised training is described in Section~\ref{nmg:section:loss}. It is the same function that was used in \cite{Katrutsa_mult} for a black-box optimization of multigrid solvers.

Next, in Section~\ref{nmg:section:restriction} we explain, that the training of $\mathcal{N}(\omega)$ is not straightforward because $A^{-1}$ is non-local, and we typically train on small problems. Namely, it is not enough to find $\omega$ that results in the small value of loss function for a fixed grid (the number of grid points controls the size of the matrix $A$). In addition to that, one needs to present a mechanism that enlarges the network when the grid is refined. If this is not done, the performance of a solver based on realistic neural networks becomes arbitrary bad for a sufficiently fine grid. To overcome this, we use the serialization of layers. Serialization performs well for some architectures, but it does not completely resolve the problem.

Concrete architectures that we test and a baseline model can be found in Section~\ref{nmg:section:models}.

As test linear operators, we use five-point, nine-point, and Mehrstellen discretizations of the Poisson equation in $D=2$ as well as anisotropic problem and a problem with mixed derivatives. The description of equations and learning results can be found in Section~\ref{nmg:section:experiments}. In short, we obtain about five times smaller spectral radius of the error propagation matrix on train set $\left(n=(2^{5}-1)^2\right)$, and from two to four times smaller spectral radius on test set $\left(n=(2^{11}-1)^2\right)$.

We conclude with an overview of related works in Section~\ref{nmg:section:related} and a short summary of the paper in Section~\ref{nmg:section:conclusion}.

All results can be reproduced (see Section~\ref{nmg:section:to_be_reproduced} for details).

\section{Multigrid method}
\label{nmg:section:multigrid_method}
We start by giving a short introduction to the multigrid method. Multigrid is a multilevel iterative method that solves linear system \eqref{nmg:eq:linear_equation} with large sparse matrix $A$. Two components crucial to fast convergence are smoother, and restriction operators \cite[Section 1.5.1]{Trottenberg_mult}.

The smoother is a cheap linear iteration \eqref{nmg:eq:linear_iteration} that effectively reduces error in a subspace $W\subset\mathbb{R}^{n}$. The overall efficiency and a subspace are controlled by the choice of matrix $N$, and the last condition is \eqref{nmg:eq:linear_iteration} ensures consistency.

The role of the restriction operator $P\in\mathbb{R}^{k\times n}, k<n$ is to perform dimension reduction. Ideally $W\perp \text{range}(P)$, so after smoothing $e^{(n+1)}\in \text{range}(P)$. That means we can project on $\text{range}(P)$, reduce the number of unknowns, and retain all information about the solution.

Having restriction and smoothing operators, we can construct a two-grid cycle:
\begin{equation}
    \label{nmg:eq:two_grid_cycle}
    \begin{split}
        &x^{(n+1/3)} = x^{(n)} + N\left(b - Ax^{(n)}\right),\\
        &\left(PAP^{T}\right)e^{(n+1/3)} = P\left(b - Ax^{(n+1/3)}\right),\\
        &x^{(n+2/3)} = x^{(n+1/3)} + P^{T}e^{(n+1/3)}\\
        &x^{(n+1)} = x^{(n+2/3)} + N\left(b - Ax^{(n+2/3)}\right).
    \end{split}
\end{equation}
The first line in \eqref{nmg:eq:two_grid_cycle} corresponds to smoothing, the second line is a coarse-grid equation, the third line is an error correction, and the last line is a smoothing again. Scheme \eqref{nmg:eq:two_grid_cycle} is preferred compare to \eqref{nmg:eq:linear_equation} because $P A P^{T}$ is a $k\times k$ matrix, that is, it is smaller, meaning cheaper to invert.

A multigrid method is a two-grid cycle, applied recursively, i.e., the two-grid cycle is used to solve the second line in \eqref{nmg:eq:two_grid_cycle}. This procedure is repeated until we reach a small enough matrix that can be inverted by direct methods, f.e., LU factorization.

In the case of a simplest geometric multigrid in $D=1$, $P$ is a convolution with stride $2$, and a kernel $\begin{bmatrix}1/2 & 1/4 & 1/2\end{bmatrix}$ (direct product of convolutions in higher dimensions). Smoother is chosen to be either some variant of damped Gauss-Seidel (first line) or damped Jacobi (second line) methods:
\begin{equation}
\label{nmg:eq:smoothers}
    \begin{split}
        &x^{(n+1)} = x^{(n)} + \omega L(A)^{-1}\left(b-Ax^{(n)}\right);\\
        &x^{(n+1)} = x^{(n)} + \omega D(A)^{-1}\left(b-Ax^{(n)}\right),
    \end{split}
\end{equation}
where $D(A)$ is a diagonal part of $A$ and $L(A)$ is a lower triangular (including the diagonal) part of $A$, and $\omega\in\mathbb{R}$ is chosen based on $A$.

Intuition about the role of smoothers and restriction operators can be gained in the simplest case of Poisson equation \cite[Chapter 13]{Saad_iter}. It can be shown that Jacobi smoother averages error. As a result, error considered as a function on a fine grid lacks high-frequency components and, as a result, can be well represented on a coarser grid. This is achieved by convolution operator $P$, which coincides with a low-pass filter combined with subsampling.

\section{Matrix-free multigrid architecture}
\label{nmg:section:multigrid_neural}
To describe our architecture, we need to introduce a few matrices. For each level $k\geq 1$ we use $A_{k}, P_{k}$ to describe matrix of linear operator and the restriction matrix. According to two-grid cycle \eqref{nmg:eq:two_grid_cycle} the following relation holds $A_{k+1} = P_{k}A_{k}P_{k}^{T}$. For level $k=1$, matrix $A_{1}$ should be given either explicitly or as a linear operator, i.e., the black-box function that computes $A_1x$ for any given $x$ suffices.

On each level $k$ we need to implement two-grid cycle \eqref{nmg:eq:two_grid_cycle} as neural network. There are four operations we need to consider: computation of residual $r_{k}\equiv b_{k} - A_{k}x_{k}$, restriction $Pr_{k}$, prolongation (interpolation) $P^{T}e_{k}$, and smoothing $N_{k} r_{k}$.

The simplest operations are restriction and prolongation that can be considered convolution with at least one stride $>1$, and a transpose to this operation.

Computation of the residual is straightforward too. Because $A_{k+1} = P_{k}A_{k}P_{k}^{T}$, any product can be computed recursively:
\begin{equation}
    \label{nmg:eq:Ax_recursion}
    A_{m}x_{m} = \left(\prod_{l=m-1,\dots,1}P_{l}\right)A_{1}\left(\prod_{l=1,\dots,m-1}P_{l}^{T}\right)x_{m}.
\end{equation}
This procedure is illustrated for $k=3$ on Fig.~\ref{nmg:fig:Ax}.

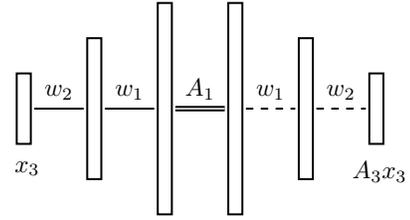
\begin{figure}[t]
    \centering
    \resizebox{0.3\textwidth}{!}{\begin{tikzpicture}
        \draw[black, thick] (0,-0.5) rectangle (0.2, 0.5);
        \node (1) [label=below:$x_3$] at (0.15, -0.5) {};
        \draw[thick] (0.25,0) -- (0.95, 0) node[above, midway] {$w_2$};
        \draw[black, thick] (1.0, -1.0) rectangle (1.2, 1.0);
        \draw[thick] (1.25,0) -- (1.95, 0) node[above, midway] {$w_{1}$};
        \draw[black, thick] (2.0, -1.5) rectangle (2.2, 1.5);
        \draw[thick, double] (2.25,0) -- (2.95, 0) node[above, midway] {$A_{1}$};
        \draw[black, thick] (3.0, -1.5) rectangle (3.2, 1.5);
        \draw[thick, dashed] (3.25,0) -- (3.95, 0) node[above, midway] {$w_{1}$};
        \draw[black, thick] (4.0, -1) rectangle (4.2, 1);
        \draw[thick, dashed] (4.25,0) -- (4.95, 0) node[above, midway] {$w_{2}$};
        \draw[black, thick] (5.0, -0.5) rectangle (5.2, 0.5);
        \node (1) [label=below:$A_3x_3$] at (5.15, -0.5) {};
    \end{tikzpicture}}
    \caption{Because $A_{i+1} = P_{i}A_{i}P_{i}^{T}$, product $A_{3} x_{3}$ can be computed as a set of convolutions (dashed lines), and transposed convolutions (solid lines), and an application of operator on the fine grid (double line); $w_{i}$ corresponds to convolution kernels.}
    \label{nmg:fig:Ax}
\end{figure}

The situation with smoothers is less straightforward. Not all smoothers can be considered in a matrix-free framework. For example, Gauss-Seidel smoother explicitly requires a lower triangular part of the matrix, which can be hard to extract. However, there is a family of smoothers, known as polynomial smoothers \cite[Section 3]{Adams_para}, that are better suited for our purposes. Polynomial smoothers take a form
\begin{equation}
    \label{nmg:eq:polynomial_smoother}
    x^{(n+1)} = x^{(n)} + p(A)\left(b-Ax^{(n)}\right),~p(A) = \sum_{i=0}^{D}\alpha_{i}A^{i},
\end{equation}
where $\alpha_{i},~i=0,\dots,D$ are parameters of the smoother chosen based on matrix $A$. Since \eqref{nmg:eq:polynomial_smoother} contains only vector-matrix products, we can apply the smoother using \eqref{nmg:eq:Ax_recursion}.

As a rule, polynomial smoothers are applied to the matrix with $1$ on diagonal, i.e., the diagonal rescaling $D(A)^{-1}$ is explicitly introduced. We hide this additional factor in convolution operation.

\begin{algorithm}[b]
\caption{Polynomial smoothing.}
\label{nmg:alg:polynomial_smoothing}
    \begin{algorithmic}[1]
        \STATE \textbf{Input:} $x_{k}$, $b_{k}$, kernels $w_{i},~i=k-1,\dots,1$ corresponding to convolutions $P_{l},~l=k-1,\dots,1$, fine grid operator $A$, kernels $\widetilde{w}_{j},~j=0,\dots,D$ that are used to compute $p(A)$.
            \\\hrulefill
            \STATE $r \leftarrow b_{k} - A_{k}(x)$ \label{nmg:alg:polynomial_smoothing:line_2}
            \FOR{$i = 1:(D+1)$}
                \STATE $x_{k} \leftarrow x_{k} + \text{conv}_{\widetilde{w}_{i-1}}\left(r\right)$ \label{nmg:alg:polynomial_smoothing:line_4}
                \STATE $r \leftarrow A_{k}(r)$ \label{nmg:alg:polynomial_smoothing:line_5}
            \ENDFOR
    \end{algorithmic}
\end{algorithm}

Algorithm~\ref{nmg:alg:polynomial_smoothing} specifies the smoothers that we use.
In lines \ref{nmg:alg:polynomial_smoothing:line_2} and \ref{nmg:alg:polynomial_smoothing:line_5}, $A_{k}(x)$ uses kernels $w_{i}$, fine-grid operator $A$ and should be computed as in \eqref{nmg:eq:Ax_recursion}, $\text{conv}_{\widetilde{w}_{i-1}}$ in line \ref{nmg:alg:polynomial_smoothing:line_4} should preserve the size of the input vector, so all strides equal $1$.

To summarize, for a given level $k$, we implement a two-grid cycle \eqref{nmg:eq:two_grid_cycle} as a convolutional neural network with the following adjustments:
\begin{enumerate}
    \item $w_{k}$, $s_{k}$ are kernel and strides (at least one stride should be $>1$) that implement convolution and transposed convolution that corresponds to $P$ and $P^{T}$;
    \item $\widetilde{w}^{(i)}_{k},i=0,\dots,D$ are kernels that are used in Algorithm~\ref{nmg:alg:polynomial_smoothing} that substitutes the first and the last lines in \eqref{nmg:eq:two_grid_cycle};
    \item all convolutions are with zero biases and without nonlinearities,
    \item any matrix-vector product is computed according to \eqref{nmg:eq:Ax_recursion} (see also Fig.~\ref{nmg:fig:Ax}).
\end{enumerate}

The whole multigrid architecture can be constructed by recursive application of two-grid layers. To imitate matrix inversion on the coarsest grid, we use a few additional convolutions.

The presence of residuals makes it hard to draw the resulting architecture. But, in general, a neural network that imitates multigrid resembles U-Net \cite{Ronneberger_unet}. The one crucial difference is that U-Net contains only one ascending and one descending branches, whereas our architecture contains an additional $\Lambda$-shaped network at each place where the residual $b_{k}-A_{k}x_{k}$ is needed.

Our approach offers the following advantages:
\begin{itemize}
    \item Currently, no major machine learning framework supports sparse-sparse matrix multiplication, so $PAP^{T}$ can not be computed efficiently.
    \item The architecture is agnostic to the sparsity pattern of $A$ and $P$ so that they can be changed easily. This can be especially useful when graph neural networks are used to learn the coarsening strategy.
    \item Since the network consists of convolution layers, one can benefit from using GPU.
    \item There is a one-to-one correspondence between some multigrid schemes and proposed architecture. This improves interpretability.
\end{itemize}

The main disadvantage is the additional operations we need to perform to compute $A_{k} x_{k}$. However, on modern GPUs, training on matrices with $n=\left(2^{5}-1\right)^2$ takes a few minutes, so the overhead seems to be justified by the overall convenience of the architecture.

\section{Loss function}
\label{nmg:section:loss}
To find optimal parameters of the neural network described in Section~\ref{nmg:section:multigrid_neural}, we introduce a loss function.

Let $x^{\star}$ be the exact solution to \eqref{nmg:eq:linear_equation}. It is known that for an arbitrary linear iterative method \eqref{nmg:eq:linear_iteration} with symmetric $I-NA$, the following is true
\begin{equation}
    \label{nmg:eq:error_propagation}
    \left\|e^{(n+1)}\right\| \leq \rho\left(I-NA\right)\left\|e^{(n)}\right\|,
\end{equation}
where $e^{(n)} = x^{(n)} - x^{\star}$ is an error, $\rho\left(I-NA\right)$ is a spectral radius, and $\left\|\cdot\right\|$ is an arbitrary norm \cite[Section 2.2.6]{Hackbusch_iter}.

Since neural multigrid architecture can be used as a linear iterative method, upper bound \eqref{nmg:eq:error_propagation} suggests that $\rho\left(I-NA\right)$ is a good loss function.

Because $\rho\left(I-NA\right)$ is not readily available, it is a custom to use an approximation or upper bound to the spectral radius. Following \cite{Katrutsa_mult}, we use Gelfand formula \cite{Kozyakin_accu}, and stochastic trace estimation \cite{Avron_rand} to derive the following approximation to the spectral radius:
\begin{equation}
    \label{nmg:eq:approximate_spectral}
    \rho(B) \simeq \rho_1(B, k, N_{\text{batch}}) \equiv \left(\frac{1}{N_{\text{batch}}}\sum_{j=1}^{N_{\text{batch}}} \left\|B^{k} z_j\right\|_{2}^{2}\right)^{1\big/2k},
\end{equation}
where $B$ is an arbitrary matrix, and each $z_j,~j=1,\dots,N_{\text{batch}}$ is a random vector with components i.i.d. according to Rademacher distribution. In all our experiments we use $k=N_{\text{batch}}=10$.
\section{Restriction on architecture for linear iterative methods}
\label{nmg:section:restriction}
Standard machine learning pipeline consists of choosing an appropriate architecture, training (supervised or unsupervised) with a given loss function, and applying trained model to unseen data \cite[Chapter 11]{Goodfellow_deep}. In this section, we argue that this approach is insufficient for training specific neural networks if we are to use them as iterative methods.

To make an argument, we consider the following boundary value problem:
\begin{equation}
    \label{nmg:eq:Laplace_point_source}
    \Delta u(x, y, z) = - \delta(x)\delta(y)\delta(z),~\left.u(x, y, z)\right|_{x^2+y^2+z^2=R^2} = 0,
\end{equation}
that is, a $3D$ Poisson equation with a point source at the origin, considered inside a sphere of radius $R$. The solution is easily obtained from Green function \cite[Section 1.10]{Jackson_classical}
\begin{equation}
    \label{nmg:eq:Laplace_point_source_solution}
    u(x, y, z) = \frac{1}{4\pi}\left(\frac{1}{\sqrt{x^2+y^2+z^2}} - \frac{1}{R}\right).
\end{equation}
One way to solve \eqref{nmg:eq:Laplace_point_source} numerically is to use finite element method (see \cite{Ciarlet_fini} for introduction). For a suitable defined mesh (for example the mesh as in Fig.~\ref{nmg:fig:two_grids} can be used), we introduce a set of piecewise linear functions $\phi_{i}(x, y, z)$ that possess cardinality property: $\phi_{i}(x_{j}, y_{j}, z_{j}) = \delta_{ij}$, where $(x_{j}, y_{j}, z_{j})$ is a fixed grid point. The solution is approximated as $u(x, y, z) = \sum_{i} \phi_{i}(x, y, z) u_{i}$, and PDE is enforced in a weak form by Petrov-Galerkin condition $ \int dxdydz~\phi_{i}(x, y, z)\left(\Delta u(x, y, z) + \delta(x)\delta(y)\delta(z)\right) = 0$. That gives us a system of linear equations \eqref{nmg:eq:linear_equation} with sparse matrix and sparse right-hand side. The sparsity of the right-hand side is illustrated by Fig.~\ref{nmg:fig:FEM}.

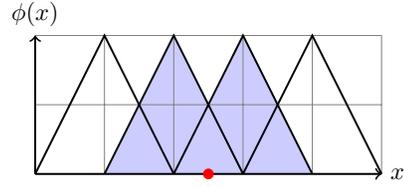
\begin{figure}[t]
    \centering
    \resizebox{0.3\textwidth}{!}{\begin{tikzpicture}
        \fill[blue!20!white] (1, 0) -- (2, 2) -- (3, 0) -- (1, 0);
        \fill[blue!20!white] (2, 0) -- (3, 2) -- (4, 0) -- (2, 0);
        \draw[step=1cm, gray, thin] (0,0) grid (5, 2);
        \draw[thick,->] (0,0) -- (5, 0) node[anchor=west] {$x$};
        \draw[thick,->] (0,0) -- (0, 2) node[anchor=south] {$\phi(x)$};
        \draw[thick, black] (0, 0) -- (1, 2) -- (2, 0);
        \draw[thick, black] (1, 0) -- (2, 2) -- (3, 0);
        \draw[thick, black] (2, 0) -- (3, 2) -- (4, 0);
        \draw[thick, black] (3, 0) -- (4, 2) -- (5, 0);
        \filldraw[red] (2.5, 0) circle (2pt);
    \end{tikzpicture}}
    \caption{When delta-function is presented as a right-hand side of a continuous problem, the weak form results in a sparse right-hand side because only a small number of (shaded) tent functions feels the presence of the source (denoted by a point).}
    \label{nmg:fig:FEM}
\end{figure}

Let $\mathcal{U}$ be a space of functions on a finite $3D$ grid with spacing $\simeq H$, and $\mathcal{N}$ be a linear neural network $\mathcal{U}\overset{\mathcal{N}}{\rightarrow} \mathcal{U}$ that acts like linear operator on space $\mathcal{U}$. Let $u_{k}\in\mathcal{U}$ be a function equals $1$ at point $k$ and $0$ at all other points. Because the grid is finite, it is possible to find a minimal radius $R_{k}$ such that all nonzero elements of $\mathcal{N}\left(u_{k}\right)$ are inside the ball with radius $R_{k}$ centered at point $k$. We define the radius of influence of a given network $\mathcal{N}$ as
\begin{equation*}
    r_{H}(\mathcal{N}) = \max_{k} R_{k}.
\end{equation*}
The example of this radius is given in Fig.~\ref{nmg:fig:two_grids} for convolution with $5\times5$ kernel.

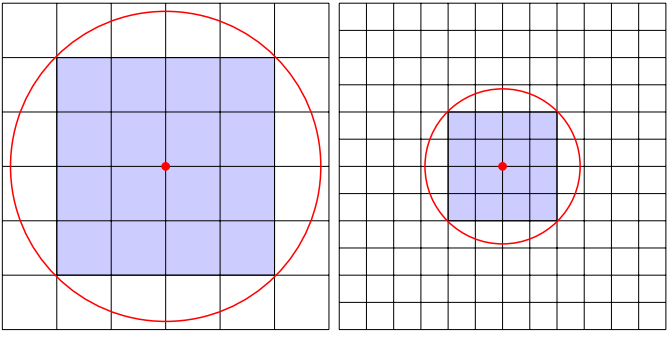
\begin{figure}[t]
    \centering
    \resizebox{0.24\textwidth}{!}{\begin{tikzpicture}
        \filldraw[fill=blue!20!white, draw=black] (1,1) rectangle (5, 5);
        \draw[step=1cm, black, thin] (0,0) grid (6,6);
        \filldraw[red] (3, 3) circle (2pt);
        \draw[red, thick] (3, 3) circle (2.85cm);
    \end{tikzpicture}}
    \resizebox{0.24\textwidth}{!}{\begin{tikzpicture}
        \filldraw[fill=blue!20!white, draw=black] (2, 2) rectangle (4, 4);
        \draw[step=0.5cm, black, thin] (0,0) grid (6,6);
        \draw[red, thick] (3, 3) circle (1.425cm);
        \filldraw[red] (3, 3) circle (2pt);
    \end{tikzpicture}}
    \caption{The figure shows how the receptive field of fixed architecture with only local layers (shaded) changes after refinement. Since convolutions are performed on discrete data, information from the dot in the middle can spread over the smaller region (enclosed by the circle) in physical space. This limits the ability to generalize for a neural network with fixed architecture.}
    \label{nmg:fig:two_grids}
\end{figure}

Now, if refinement is performed and the architecture of the network does not contain dense layers, the radius of influence shrinks as explained in the same Fig.~\ref{nmg:fig:two_grids}. This fact is used to prove the following statement.

\begin{prop}
Let $A_{H}$ be a matrix of linear problem \eqref{nmg:eq:Laplace_point_source} obtained using finite element method on a given grid with spacing $\simeq H$. Let $\mathcal{N}$ be a neural network, that consists on finite number of (local) convolutional layers\footnote{We exclude nonlocal convolutions based on graph Laplacian as in \cite{Bruna_spectral}.}, and used as $N$ in linear iterative method \eqref{nmg:eq:linear_iteration}. Suppose that the network has been trained to provide a good convergence for grid $H$, that is, $\rho(I-\mathcal{N}_{H}A_{H})=\epsilon\ll1$. It is always possible to find a grid with spacing $\simeq h<H$ such that $\rho(I-\mathcal{N}_{h}A_{h})$ is arbitrary close to $1$.
\end{prop}
\begin{proof}
Without loss of generality we can assume that for grid $H$ the radius of influence $r_{H}(\mathcal{N})$ is smaller than a grid size in a physical space, which is $R$ for our problem. For $h = H\big/ 2^{p}, p>1$ the radius of influence is $r_{h}(\mathcal{N}) = r_{H}(\mathcal{N})/2^{p}$. Let $b_{h}$ be a discrete right hand side corresponding to a delta function in equation~\eqref{nmg:eq:Laplace_point_source}. If we start from zero initial guess $x^{(0)} = 0$, an estimation to the initial error in $L_{2}$ norm reads
\begin{equation}
    h^3\left\|e^{(0)}\right\|_{2}^{2} \simeq 4 \pi\int_{0}^{R} dr~ r^2 u(r)^2 = R\big/(12\pi),
\end{equation}
and a lower bound on error for step $n=K$ ($Kr_{h}\left(\mathcal{N}\right)<R$) reads
\begin{equation}
    \label{nmg:eq:lower_bound}
    \begin{split}
        h^3\left\|e^{(K)}\right\|_{2}^{2} \geq 4 \pi\int_{Kr_{h}\left(\mathcal{N}\right)}^{R} dr~ r^2 u(r)^2 = \\ = \left(R\big/(12\pi)\right)\left(1 - \frac{K r_{h}\left(\mathcal{N}\right)}{R}\right)^3.
    \end{split}
\end{equation}
To derive equation~\eqref{nmg:eq:lower_bound}, we assumed that our iterative method recovers the exact solution for all points that the network reached. Note that this argument is valid only because $b_{h}$ is a sparse vector.

From \eqref{nmg:eq:error_propagation} we conclude
\begin{equation}
    \rho(I-\mathcal{N}_{h}A)^{K} \geq \left\|e^{(K)}\right\|_{2} \big/ \left\|e^{(0)}\right\|_{2} \geq \left(1 - \frac{K r_{h}\left(\mathcal{N}\right)}{R}\right)^{3/2}.
\end{equation}
\end{proof}
Because $r_{h}$ can be arbitrary small for sufficiently small $h=H/2^{p}$, the expression in the brackets above can be arbitrary close to $1$, which signifies arbitrary slow convergence.

\begin{rem}
    The proposition above holds for networks that consist of (local) convolutional layers. We exclude networks with dense and nonlocal layers because they require $\simeq O(N^2)$ ($N$ is a number of inputs) flops, which is unacceptable for iterative methods. On the other hand, convolutional neural networks require $\simeq O(N)$ flops and can be applied on grids with different sizes and geometries.
\end{rem}

\begin{rem}
    Our ``radius of influence'' is similar to the ``domain of dependence'' used to analyze convergence of numerical methods \cite[Section 10.7]{LeVeque_fini}. Also, there is an evident parallel with CFL condition \cite{CFL}.
\end{rem}

\begin{cor}
Let $A_{H}$ be a matrix of linear problem \eqref{nmg:eq:Laplace_point_source} obtained using finite element method on a given grid with spacing $\simeq H$. Let $\mathcal{N}$ be a neural network, that consists of finite number of (local) convolutional layers. It is not possible to have $\left\|I-\mathcal{N}_{h}A_{h}\right\|\leq \epsilon \ll 1$, with $\epsilon$ independent on $h<H$. In other words, it is impossible to uniformly approximate inverses to $A_{h}$ using fixed architecture with local layers.
\end{cor}
\begin{proof}
Since $\rho(I-\mathcal{N}_{h}A_{h})\leq \left\|I-\mathcal{N}_{h}A_{h}\right\|$ for any matrix norm, the statement can be proven by contradiction.
\end{proof}

\begin{rem}
    It is crucial that operator $A_{h}^{-1}$ is nonlocal. For example, $A_{h}$ from the statement is equivalent to $\mathcal{N}$ with a single convolutional layer.
\end{rem}

\begin{rem}
    It is known that neural network can approximate arbitrary nonlinear operator \cite{Chen_univ}. The corollary above does not contradict this result because it is restricted to neural networks with a finite number of layers.
\end{rem}

\section{Architectures and a baseline solver}
\label{nmg:section:models}
Architecture that we propose in Section~\ref{nmg:section:multigrid_neural} is a convolutional neural network. We want to train this architecture on small linear problems with a number of variables $n \simeq 2^{10}$ and apply it on large linear problems with $n \simeq 2^{20}$. According to the result in the previous section, it is necessary to enlarge the network when we refine the grid. The simplest strategy is a serialization of layers. By serialization, we mean that an additional layer uses parameters from a previous layer. Here we formulate a few concrete architectures that we are going to compare in Section~\ref{nmg:section:experiments}.

\subsection{LMG}
As a baseline model we use multigrid with linear interpolation and two pre-smoothing and two post-smoothing Jacobi sweeps (second line of equation~\eqref{nmg:eq:smoothers}) with $\omega=4/5$ (this $\omega$ is optimal for five-point discretization of Poisson equation in $2D$ \cite[Section 2.1.2]{Trottenberg_mult}). Linear interpolation means that $P$ corresponds to convolution with strides $(2, 2)$ with the kernel
\begin{equation}
\label{nmg:eq:linear_interpolation}
    k_{\text{linear}} =
    \frac{1}{2}\left[\begin{matrix}
        1/4 & 1/2 & 1/4 \\
        1/2 & 1 & 1/2 \\
        1/4 & 1/2 & 1/4
    \end{matrix}\right].
\end{equation}

\subsection{s$1$MG(rs)}
The name of the model derived from the fact that it is a neural multigrid (MG) architecture with a single serialized layer (s1), which contain adjustable restriction and smoothing operators (rs), with weights $w$ and $\widetilde{w}$ respectively. To have the same number of floating-point operations as a baseline model, we use smoothing (Algorithm~\ref{nmg:alg:polynomial_smoothing}) with $D=0$. Both $w$ and $\widetilde{w}$ represents kernels of sizes $3\times3$, which initially coincide with linear interpolation~\eqref{nmg:eq:linear_interpolation}. Convolutional layer with kernel $w$ has strides $(2, 2)$, and the layer with kernel $\widetilde{w}$ has strides $(1, 1)$. For this model, we use exact matrix inversion as a coarse-grid correction. This is possible because we can always stack enough layers to have a single unknown on a coarse grid for considered model problems.

\subsection{s$1$MG(s)}
This model is the same as the previous one but with two differences. First, the restriction operator is fixed to be linear interpolation \eqref{nmg:eq:linear_interpolation}, and only the smoothing operator is learned. Second, we explicitly incorporate diagonal rescaling with $D(A)^{-1}$ on each level. This is possible because restriction operators are fixed, so all diagonal can be computed in advance.

\subsection{s$3$MG(s)}
The model is the same as a previous one, but now we train three distinct layers with $\widetilde{w}_{1}$, $\widetilde{w}_{2}$, $\widetilde{w}_{3}$. The serialization is performed as follows:
\begin{equation}
    \begin{split}
        &\text{layer }1:~\widetilde{w}_{1};~
        \text{layer }2:~\widetilde{w}_{2};~
        \text{layer }3:~\widetilde{w}_{3};\\
        &\text{layer }4:~\widetilde{w}_{1};~
        \text{layer }5:~\widetilde{w}_{2};~
        \text{layer }6:~\widetilde{w}_{3};\\
        &\text{layer }7:~\widetilde{w}_{1};~\dots
    \end{split}
\end{equation}

\subsection{U-Net}

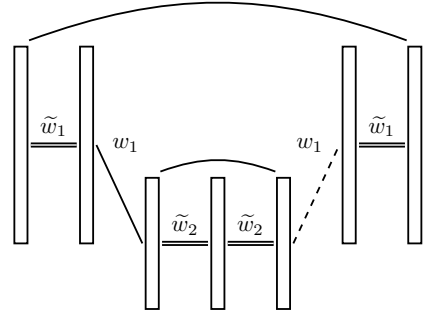
\begin{figure}[t]
    \centering
    \resizebox{0.3\textwidth}{!}{\begin{tikzpicture}
        \draw[black, thick] (0,-1.5) rectangle (0.2, 1.5);
        \node[] at (0.1, 1.55) (a1) {};
        \draw[thick, double] (0.25,0) -- (0.95, 0) node[above, midway] {$\widetilde{w}_1$};
        \draw[black, thick] (1.0,-1.5) rectangle (1.2, 1.5);
        \draw[thick] (1.25,0) -- (1.95, -1.5);
        \node[] at (1.7, 0) {$w_{1}$};
        \draw[black, thick] (2.0,-2.5) rectangle (2.2, -0.5);
        \node[] at (2.1, -0.45) (b1) {};
        \draw[thick, double] (2.25,-1.5) -- (2.95, -1.5) node[above, midway] {$\widetilde{w}_2$};
        \draw[black, thick] (3.0,-2.5) rectangle (3.2, -0.5);
        \draw[thick, double] (3.25,-1.5) -- (3.95, -1.5) node[above, midway] {$\widetilde{w}_2$};
        \draw[black, thick] (4.0,-2.5) rectangle (4.2, -0.5);
        \node[] at (4.1, -0.45) (b2) {};
        \draw[thick] (b1) to[out=20,in=160] (b2);
        \draw[thick, dashed] (4.25, -1.5) -- (4.95, 0);
        \node[] at (4.5, 0) {$w_{1}$};
        \draw[black, thick] (5.0,-1.5) rectangle (5.2, 1.5);
        \draw[thick, double] (5.25,0) -- (5.95, 0) node[above, midway] {$\widetilde{w}_1$};
        \draw[black, thick] (6.0,-1.5) rectangle (6.2, 1.5);
        \node[] at (6.1, 1.55) (a2) {};
        \draw[thick] (a1) to[out=20,in=160] (a2);
    \end{tikzpicture}}
    \caption{U-Net architecture with two layers. Convolutions with all strides equal $1$ are denoted by double lines (they correspond to smoothing in multigrid architecture), the single line represents convolution with at least one stride $>1$, dashed line is a transpose to this operation, a curved line is a skip connection (copy and add).}
    \label{nmg:fig:unet}
\end{figure}

This is an attempt to reproduce results from \cite{Hsieh_lear}.\footnote{Which is nontrivial because the code is absent and the architecture of the model is unspecified. We cannot also use results from the article because they are scarce, and authors measure performance relative to the multigrid method, which they did not bother to describe in detail.} We use architecture presented in Fig.~\ref{nmg:fig:unet}, but with $5$ layers. U-Net is used as $N$ in linear iteration \eqref{nmg:eq:linear_interpolation}. Parameters for all layers are distinct, and no serialization is performed.

\subsection{fMG}

This is another model without serialization. We use $5$ layers with distinct restriction $w$ and smoothing $\widetilde{w}$ operators. Two convolutions are used as a coarse grid correction. All kernels are initialised as bilinear interpolation \eqref{nmg:eq:linear_interpolation}.

\section{Experiments}
\label{nmg:section:experiments}
We start by defining the model equation and then comment on the performance of the models.  All of the equations below correspond to the following boundary value problem
\begin{equation}
    \begin{split}
    \left(a\frac{\partial^2}{\partial x^2} + b\frac{\partial^2}{\partial y^2} + c\frac{\partial^2}{\partial x\partial y}\right)u(x, y) = f(x, y),\\
    x,y\in\left(0, 1\right)^2\equiv\Gamma,~\left.u(x, y)\right|_{\partial \Gamma} = 0,
    \end{split}
\end{equation}
that is, a second-order equation with Dirichlet boundary conditions. For all discrete problems, we also perform a Jacobi preconditioning step $A \rightarrow D(A)^{-1/2} A D(A)^{-1/2}$ \cite[Section 3.1]{Wathen_prec}.
\subsection{Model equations}
\subsubsection{Poisson equation}
Here $a = b = -1$, $c = 0$, and the corresponding kernels are
\begin{equation}
\label{nmg:eq:poisson_5}
    k_{\text{P(5)}} = \left[\begin{matrix}
    0 & -1/4 & 0\\
    -1/4 & 1 & -1/4\\
    0 & -1/4 & 0
    \end{matrix}\right],
\end{equation}
\begin{equation}
\label{nmg:eq:poisson_9}
    k_{\text{P(9)}} = \left[\begin{matrix}
    0 & 0 & 1/60 & 0 & 0\\
    0 & 0 & -4/15 & 0 & 0\\
    1/60 & -4/15 & 1 & -4/15 & 1/60\\
    0 & 0 & -4/15 & 0 & 0\\
    0 & 0 & 1/60 & 0 & 0\\
    \end{matrix}\right],
\end{equation}
\begin{equation}
\label{nmg:eq:poisson_Mehrstellen}
    k_{\text{P(M)}} = \left[\begin{matrix}
    -1/20 & -1/5 & -1/20\\
    -1/20 & 1 & -1/20\\
    -1/20 & -1/5 & -1/20\\
    \end{matrix}\right],
\end{equation}
which correspond to second-order, and two distinct fourth-order schemes. The last discretization is known as Mehrstellen and can be used to construct sixth-order accurate discretization for sufficiently smooth right-hand side and boundary data \cite{Rosser_nine}.
\subsubsection{Anisotropic Poisson equation}
In this case $a=-\epsilon$, $b=-1$, $c=0$ and the kernel reads
\begin{equation}
\label{nmg:eq:anisotropic}
    k_{\text{A}} = \left[\begin{matrix}
    0 & -1/(2+2\epsilon) & 0\\
    -\epsilon/(2+2\epsilon) & 1 & -\epsilon/(2+2\epsilon)\\
    0 & -1/(2+2\epsilon) & 0
    \end{matrix}\right],
\end{equation}
and we use $\epsilon = 2$ and $\epsilon = 10$.
\subsubsection{Mixed derivative}
Here $a = b = -1$ and $c = 2\tau$. The kernel is
\begin{equation}
\label{nmg:eq:mixed}
    k_{\text{M}} = \left[\begin{matrix}
    -\tau/8 & -1/4 & \tau/8\\
    -1/4 & 1 & -1/4\\
    \tau/8 & -1/4 & -\tau/8
    \end{matrix}\right],
\end{equation}
and we test for $\tau=1/4$ and $\tau=3/4$.
\subsection{Results}
Results are gathered in Tables \ref{nmg:table:Poisson_5}--\ref{nmg:table:Mixed_3_4}. Each table contains $\rho(I-\mathcal{N}A)$ approximated by equation \eqref{nmg:eq:approximate_spectral} for a given architecture $\mathcal{N}$, $J$ fixes the number of grid points along each direction $n_x = (2^J-1)$, $n_y = (2^J-1)$, and a total number of points $n = n_x n_y$. Symbol ``$-$'' means that $\rho_{1}\left(\left(I-\mathcal{N}A\right), 10,~10\right)\geq1$ (see \eqref{nmg:eq:approximate_spectral})\footnote{This fact does not automatically mean that the actual spectral radius is greater than one. It might as well be merely close to one. In any case $\rho_{1}\left(\left(I-\mathcal{N}A\right), 10,~10\right)\geq1$ implies a significant deterioration of the solver.}. Each model that uses serialization applied with $J$ layers, U-Net and fMG both contain $\leq5$ layers for all grids. The training is done for $J\le 5$, then we test for $J\in\left[6, 11\right]$.
\subsubsection{Poisson equation}
\begin{table}[t]
    \caption{$k_{\text{P(5)}}$, equation~\eqref{nmg:eq:poisson_5}, $\rho(I-\mathcal{N}A)$}
    \begin{center}
        \begin{tabular}{c|c|c|c|c|c|c}
        $J$ & LMG & s$1$MG(rs) & s$1$MG(s) & s$3$MG(s) & U-Net & fMG\\
        \hline
        $3$ & $0.11$ & $0.047$ & $0.046$ & $0.028$ & $0.50$ & $0.031$\\
        \hline
        $4$ & $0.13$ & $0.051$ & $0.050$ & $0.035$ & $0.54$ & $0.034$\\
        \hline
        $5$ & $0.15$ & $0.057$ & $0.058$ & $0.041$ & $0.58$ & $0.037$\\
        \hline
        \hline
        $6$ & $0.16$ & $0.19$ & $0.066$ & $0.050$ & $0.92$ & $0.37$\\
        \hline
        $7$ & $0.17$ & $0.58$ & $0.073$ & $0.059$ & $-$ & $0.80$\\
        \hline
        $8$ & $0.19$ & $-$ & $0.080$ & $0.065$ & $-$ & $-$\\
        \hline
        $9$ & $0.20$ & $-$ & $0.088$ & $0.073$ & $-$ & $-$\\
        \hline
        $10$ & $0.21$ & $-$ & $0.094$ & $0.083$ & $-$ & $-$\\
        \hline
        $11$ & $0.23$ & $-$ & $0.10$ & $0.092$ & $-$ & $-$\\
        \end{tabular}
        \label{nmg:table:Poisson_5}
    \end{center}
\end{table}
\begin{table}[b]
    \caption{$k_{\text{P(9)}}$, equation~\eqref{nmg:eq:poisson_9}, $\rho(I-\mathcal{N}A)$}
    \begin{center}
        \begin{tabular}{c|c|c|c|c|c|c}
            $J$ & LMG & s$1$MG(rs) & s$1$MG(s) & s$3$MG(s) & U-Net & fMG\\
            \hline
            $3$ & $0.16$ & $0.058$ & $0.069$ & $0.042$ & $0.53$ & $0.030$\\
            \hline
            $4$ & $0.22$ & $0.063$ & $0.073$ & $0.041$ & $0.58$ & $0.031$\\
            \hline
            $5$ & $0.25$ & $0.070$ & $0.079$ & $0.049$ & $0.62$ & $0.038$\\
            \hline
            \hline
            $6$ & $0.28$ & $0.081$ & $0.088$ & $0.086$ & $-$ & $0.54$\\
            \hline
            $7$ & $0.30$ & $0.33$ & $0.096$ & $0.11$ & $-$ & $0.89$\\
            \hline
            $8$ & $0.32$ & $0.82$ & $0.10$ & $0.12$ & $-$ & $-$\\
            \hline
            $9$ & $0.35$ & $-$ & $0.11$ & $0.13$ & $-$ & $-$\\
            \hline
            $10$ & $0.37$ & $-$ & $0.12$ & $0.14$ & $-$ & $-$\\
            \hline
            $11$ & $0.40$ & $-$ & $0.13$ & $0.15$ & $-$ & $-$\\
        \end{tabular}
        \label{nmg:table:Poisson_9}
    \end{center}
\end{table}
\begin{table}[tb]
    \caption{$k_{\text{P(M)}}$, equation~\eqref{nmg:eq:poisson_Mehrstellen}, $\rho(I-\mathcal{N}A)$}
    \begin{center}
        \begin{tabular}{c|c|c|c|c|c|c}
            $J$ & LMG & s$1$MG(rs) & s$1$MG(s) & s$3$MG(s) & U-Net & fMG\\
            \hline
            $3$ & $0.073$ & $0.030$ & $0.036$ & $0.022$ & $0.40$ & $0.017$\\
            \hline
            $4$ & $0.094$ & $0.034$ & $0.041$ & $0.028$ & $0.44$ & $0.019$\\
            \hline
            $5$ & $0.11$ & $0.041$ & $0.048$ & $0.02$ & $0.60$ & $0.022$\\
            \hline
            \hline
            $6$ & $0.12$ & $0.13$ & $0.058$ & $0.042$ & $0.96$ & $0.52$\\
            \hline
            $7$ & $0.13$ & $0.40$ & $0.065$ & $0.051$ & $-$ & $0.99$\\
            \hline
            $8$ & $0.14$ & $0.78$ & $0.071$ & $0.057$ & $-$ & $-$\\
            \hline
            $9$ & $0.15$ & $0.99$ & $0.077$ & $0.063$ & $-$ & $-$\\
            \hline
            $10$ & $0.16$ & $-$ & $0.083$ & $0.069$ & $-$ & $-$\\
            \hline
            $11$ & $0.17$ & $-$ & $0.090$ & $0.076$ & $-$ & $-$\\
        \end{tabular}
        \label{nmg:table:Mehrstellen}
    \end{center}
\end{table}
(Tables \ref{nmg:table:Poisson_5}--\ref{nmg:table:Mehrstellen}) For all discretizations of the Poisson equation, we can see that architectures U-Net, fMG, and s$1$MG(rs) fail to provide a good solver for $J\ge 6$.

Presumably, the spectral radius of error propagation matrices corresponding to U-Net and fMG architectures deteriorates because neural networks have fixed sizes.

This explanation does not work for s$1$MG(rs) because of the serialization performed. We can conjure that because both restriction and smoothing operators are optimized, s$1$MG(rs) is getting tuned to the spectrum of the matrix with $n=(2^5-1)^2$, since the spectrum changes when $J$ increases, the solver ceases to be efficient. It is evident from other examples that the naive serialization does not seem to work when both restriction and smoothing operators are optimized.

The only two solvers (besides a baseline model) that retain their efficiency are s$1$MG(s) and s$3$MG(s). The latter is better than the former for five-point \eqref{nmg:eq:poisson_5} and Mehrstellen \eqref{nmg:eq:poisson_Mehrstellen} discretizations, but for the long stencil \eqref{nmg:eq:poisson_9} s$1$MG(s) is superior.

We can conclude that for the Poisson equation, U-Net is the weakest model, fMG and s$1$MG(rs) fail to generalize on the test set, and both s$1$MG(s) and s$3$MG(s) can generalize and outperform a baseline model on a test set.

\subsubsection{Anisotropic Poisson equation}
\begin{table}[b]
    \caption{$k_{\text{A}}$, equation~\eqref{nmg:eq:anisotropic}, $\epsilon=2$}
    \begin{center}
        \begin{tabular}{c|c|c|c|c|c|c}
        $J$ & LMG & s$1$MG(rs) & s$1$MG(s) & s$3$MG(s) & U-Net & fMG\\
        \hline
        $3$ & $0.23$ & $0.071$ & $0.076$ & $0.047$ & $0.65$ & $0.060$\\
        \hline
        $4$ & $0.29$ & $0.076$ & $0.085$ & $0.048$ & $0.70$ & $0.067$\\
        \hline
        $5$ & $0.31$ & $0.084$ & $0.093$ & $0.054$ & $0.76$ & $0.075$\\
        \hline
        \hline
        $6$ & $0.33$ & $0.31$ & $0.10$ & $0.066$ & $-$ & $0.39$\\
        \hline
        $7$ & $0.36$ & $0.74$ & $0.11$ & $0.089$ & $-$ & $0.74$\\
        \hline
        $8$ & $0.38$ & $-$ & $0.12$ & $0.10$ & $-$ & $-$\\
        \hline
        $9$ & $0.41$ & $-$ & $0.13$ & $0.11$ & $-$ & $-$\\
        \hline
        $10$ & $0.44$ & $-$ & $0.15$ & $0.12$ & $-$ & $-$\\
        \hline
        $11$ & $0.47$ & $-$ & $0.16$ & $0.13$ & $-$ & $-$\\
        \end{tabular}
        \label{nmg:table:Anisotropic_2}
    \end{center}
\end{table}
\begin{table}[t]
    \caption{$k_{\text{A}}$, equation~\eqref{nmg:eq:anisotropic}, $\epsilon=10$}
    \begin{center}
        \begin{tabular}{c|c|c|c|c|c|c}
        $J$ & LMG & s$1$MG(rs) & s$1$MG(s) & s$3$MG(s) & U-Net & fMG\\
        \hline
        $3$ & $0.59$ & $0.40$ & $0.44$ & $0.42$ & $0.91$ & $0.44$\\
        \hline
        $4$ & $0.73$ & $0.42$ & $0.47$ & $0.42$ & $0.99$ & $0.47$\\
        \hline
        $5$ & $0.81$ & $0.49$ & $0.53$ & $0.49$ & $-$ & $0.53$\\
        \hline
        \hline
        $6$ & $0.88$ & $-$ & $0.59$ & $0.53$ & $-$ & $-$\\
        \hline
        $7$ & $0.94$ & $-$ & $0.63$ & $0.57$ & $-$ & $-$\\
        \hline
        $8$ & $-$ & $-$ & $0.68$ & $0.61$ & $-$ & $-$\\
        \hline
        $9$ & $-$ & $-$ & $0.73$ & $0.65$ & $-$ & $-$\\
        \hline
        $10$ & $-$ & $-$ & $0.78$ & $0.70$ & $-$ & $-$\\
        \hline
        $11$ & $-$ & $-$ & $0.83$ & $0.75$ & $-$ & $-$\\
        \end{tabular}
        \label{nmg:table:Anisotropic_10}
    \end{center}
\end{table}
(Tables \ref{nmg:table:Anisotropic_2}, \ref{nmg:table:Anisotropic_10}) For equation \eqref{nmg:eq:anisotropic}, the trend is largely the same. That is, fMG, s$1$MG(rs) and U-Net lose their efficiency, s$1$MG(s) and s$3$MG(s) are robust and outperform a baseline model.

It is instructive to discuss results for anisotropic equation with $\epsilon=10$. First, we can see that all solvers are relatively inefficient. The reason is a full coarsening that we applied. If one uses semicoarsening instead, the results would be the same as for the isotropic Poisson equation. Because of the full coarsening, the U-Net solver fails already on a train set. If one further increases $\epsilon$, our networks would not be able to provide efficient solvers unless strides are chosen appropriately. If strides and sizes of filters are considered as hyperparameters, it should be possible to apply Bayesian optimization \cite{Shahriari_taki}, reinforcement learning \cite{Li_hype}, or genetic programming \cite{Schitt_cons} to construct optimal solver.

\subsubsection{Mixed derivative}
\begin{table}[b]
    \caption{$k_{\text{M}}$, equation~\eqref{nmg:eq:mixed}, $\tau=1/4$}
    \begin{center}
        \begin{tabular}{c|c|c|c|c|c|c}
        $J$ & LMG & s$1$MG(rs) & s$1$MG(s) & s$3$MG(s) & U-Net & fMG\\
        \hline
        $3$ & $0.11$ & $0.040$ & $0.049$ & $0.029$ & $0.49$ & $0.043$\\
        \hline
        $4$ & $0.14$ & $0.049$ & $0.056$ & $0.032$ & $0.54$ & $0.048$\\
        \hline
        $5$ & $0.15$ & $0.061$ & $0.065$ & $0.037$ & $0.58$ & $0.054$\\
        \hline
        \hline
        $6$ & $0.17$ & $0.27$ & $0.073$ & $0.045$ & $0.88$ & $0.34$\\
        \hline
        $7$ & $0.18$ & $0.67$ & $0.081$ & $0.052$ & $-$ & $0.78$\\
        \hline
        $8$ & $0.19$ & $-$ & $0.088$ & $0.062$ & $-$ & $0.98$\\
        \hline
        $9$ & $0.21$ & $-$ & $0.095$ & $0.072$ & $-$ & $-$\\
        \hline
        $10$ & $0.22$ & $-$ & $0.10$ & $0.083$ & $-$ & $-$\\
        \hline
        $11$ & $0.24$ & $-$ & $0.11$ & $0.092$ & $-$ & $-$\\
        \end{tabular}
        \label{nmg:table:Mixed_1_4}
    \end{center}
\end{table}
\begin{table}[t]
    \caption{$k_{\text{M}}$, equation~\eqref{nmg:eq:mixed}, $\tau=3/4$}
    \begin{center}
        \begin{tabular}{c|c|c|c|c|c|c}
        $J$ & LMG & s$1$MG(rs) & s$1$MG(s) & s$3$MG(s) & U-Net & fMG\\
        \hline
        $3$ & $0.24$ & $0.097$ & $0.15$ & $0.055$ & $0.51$ & $0.062$\\
        \hline
        $4$ & $0.35$ & $0.11$ & $0.16$ & $0.071$ & $0.56$ & $0.069$\\
        \hline
        $5$ & $0.41$ & $0.12$ & $0.19$ & $0.087$ & $0.61$ & $0.082$\\
        \hline
        \hline
        $6$ & $0.46$ & $0.24$ & $0.23$ & $0.19$ & $0.79$ & $0.33$\\
        \hline
        $7$ & $0.50$ & $-$ & $0.25$ & $0.25$ & $0.99$ & $0.72$\\
        \hline
        $8$ & $0.54$ & $-$ & $0.28$ & $0.28$ & $-$ & $0.98$\\
        \hline
        $9$ & $0.59$ & $-$ & $0.31$ & $0.30$ & $-$ & $-$\\
        \hline
        $10$ & $0.63$ & $-$ & $0.33$ & $0.33$ & $-$ & $-$\\
        \hline
        $11$ & $0.68$ & $-$ & $0.36$ & $0.35$ & $-$ & $-$\\
        \end{tabular}
        \label{nmg:table:Mixed_3_4}
    \end{center}
\end{table}
(Tables \ref{nmg:table:Mixed_1_4}-\ref{nmg:table:Mixed_3_4}) Equation with mixed derivative changes type from elliptic to hyperbolic when $\tau$ crosses $1$. It is interesting to look at how our models behave when $\tau$ approach $1$.

For $\tau=1/4$ architectures s$1$MG(s), s$3$MG(s) produces more efficient solvers than the standard multigrid with two Jacobi sweeps. On the other hand, U-Net is of no use even on the test set, and fMG and s$1$MG(rs) deteriorate rapidly for $J>5$.

We can see that for $\tau=3/4$ s$3$MG(s) performs substantially better than s$1$MG(s) on the train set. However, on the test set, it results in only a marginally smaller spectral radius. This means that the training and serialization strategies are not ideal. It should be possible to use additional parameters of s$3$MG(s) more efficiently. Other architectures behave similarly to the case $\tau=1/4$.
\section{Reproducibility}
\label{nmg:section:to_be_reproduced}
To ensure complete reproducibility, we share a set of Jupyter notebooks that contain all models, linear equations, and training loops: \textit{https://github.com/VLSF/nmg}.

\section{Related work}
\label{nmg:section:related}
Here we discuss a few related attempts to improve the multigrid method with machine learning tools. In the already mentioned paper \cite{Katrutsa_mult}, authors use stochastic gradient-based optimization to learn optimal multigrid solvers. This work roughly corresponds to architecture s$1$MG(rs), and our training strategy is exactly the same as in \cite{Katrutsa_mult}. From the results, (Tables \ref{nmg:table:Poisson_5}--\ref{nmg:table:Mixed_3_4}) we can conclude that simultaneous optimization of restriction and smoothing operators does not lead to a robust solver.

The other multilevel solver that we tried is a U-Net from \cite{Hsieh_lear}. As we pointed in Section \ref{nmg:section:models}, we cannot be sure that we reproduce results from \cite{Hsieh_lear} because the paper contains omissions. However, U-Net architecture from Fig.~\ref{nmg:fig:unet} fails on the test set and even unable to work on the train set for the anisotropic Poisson equation. It is not hard to see that U-Net architecture is just a slightly generalized filtering preconditioner \cite{Tong_mult}. Given that, the whole scheme from \cite{Hsieh_lear} is a generalized Richardson iteration for the preconditioned system. An optimal spectral radius for the preconditioner Richardson method is $\left(\kappa(NA)-1\right)\big/\left(\kappa(NA)+1\right)$, where $\kappa(NA)$ is a condition number, so to match multigrid $\kappa(NA)$ should be about $1.5$. This means $N$ should be much better than (optimal) Schwarz preconditioners for Poisson equation \cite{Zhang_mult}.

Two articles that firmly demonstrate that machine learning is a valuable tool for the construction of multigrid solvers are \cite{Greenfeld_learning} (geometric multigrid), \cite{Luz_learning} (algebraic multigrid). In both cases, the authors take Gauss-Seidel smoother and focus on restriction weights. In some sense, our contribution is complementary because we focus on finding optimal smoothers and use bilinear interpolation as $P^T$. We can speculate that in both cases, scalable solvers are obtained in part because authors utilize ready-made coarsening strategies and robust smoother. Namely, in \cite{Greenfeld_learning} a strategy from algebraic multigrid (AMG) is used to restore the solution on the fine grid in such a way that $b - Ax = 0$ for red points in a red-black pattern, and in \cite{Luz_learning} authors completely rely on AMG coarsening strategy. An attractive alternative would be to use kriging to perform coarsening as explained in \cite{Gottschalk_coar}.

Other related areas are Bootstrap AMG \cite{Brandt_Bootstrap} and optimization based on local Fourier analysis \cite{Wienands_prac}, \cite{Schmitt_opti}.

\section{Conclusion}
\label{nmg:section:conclusion}
We introduce a convenient architecture that represents multigrid as a convolutional neural network. Using the simple $3D$ Poisson equation, we argue that the training of linear solver should be supplemented by a mechanism that enlarges the network's size. The simplest possible solution based on serialization of layers performs well, i.e., result in a robust solver competitive with a baseline model, but only for some architectures. Sadly, serialization does not work for the most promising architecture that combines optimization of smoothing and restriction operators. In our opinion, this problem can be solved either by a modification of the loss function or by an introduction of an additional mechanism that assembles multigrid based on pretrained layers. This is the focus of our current research.

\end{document}